\documentclass[11pt]{article}

\usepackage[a4paper,margin=1in]{geometry}
\usepackage{microtype}
\usepackage{amsmath,amssymb,amsfonts,amsthm,mathtools}
\usepackage{enumitem}
\usepackage{booktabs}
\usepackage{graphicx}
\usepackage{float}
\usepackage{tikz}
\usepackage{pgfplots}
\pgfplotsset{compat=1.18}
\usepackage{xcolor}

\usepackage{hyperref}
\usepackage[nameinlink,noabbrev]{cleveref}
\usepackage{xurl}

\newtheorem{theorem}{Theorem}[section]
\newtheorem{lemma}[theorem]{Lemma}
\newtheorem{proposition}[theorem]{Proposition}

\theoremstyle{definition}
\newtheorem{definition}[theorem]{Definition}
\theoremstyle{remark}
\newtheorem{remark}[theorem]{Remark}

\newcommand{\CC}{\mathbb{C}}

\newcommand{\calH}{\mathcal{H}}
\newcommand{\calP}{\mathcal{P}}

\DeclareMathOperator{\ind}{ind}

\title{Zero-freeness of a multivariate monomer--dimer--cycle polynomial on bounded-degree graphs}
\author{Paula M.\ S.\ Fialho\footnote{paula.fialho@dcc.ufmg.br. Acknowledges support from FAPEMIG.} \and Gabriel Coutinho\footnote{gabriel@dcc.ufmg.br. Acknowledges support from CNPq and FAPEMIG.}}
\date{\today}

\begin{document}
\maketitle

\begin{abstract}
We initiate the study of a multivariate graph polynomial $\Phi_G(x,y,z)$ that interpolates between classical counting polynomials for matchings and for cycle structures arising in the Harary--Sachs expansion of the characteristic polynomial. We focus on analytic properties and computational consequences. Our main contribution is an explicit, degree-uniform zero-free region for $\Phi_G$ on bounded-degree graphs, obtained via the Fern\'andez--Procacci convergence criterion for abstract polymer gases.
\end{abstract}

\section{Introduction}

The study of graph polynomials and their zeros is a central theme in algebraic combinatorics, with deep connections to statistical mechanics, probability theory, and theoretical computer science.
In this work, we propose the study of a new multivariate polynomial that interpolates between the matching polynomial and the characteristic polynomial of the adjacency matrix of a graph.

Let $G=(V,E)$ be a finite simple graph with $|V|=n$ and maximum degree $\Delta<\infty$.
We recall that a \emph{simple $k$-cycle} is a cycle on $k$ distinct vertices (and $k$ edges), with $k\ge 3$.

\begin{definition}[Sesquivalent (or elementary) subgraphs]\label{def:sesq}
A spanning subgraph $H=(V,E_H)$ of $G$ is \emph{sesquivalent} if every connected component of $H$ is either
an isolated vertex, an edge (a copy of $K_2$), or a simple cycle.
We denote by $\calH(G)$ the set of all sesquivalent subgraphs of $G$.
\end{definition}

For $H\in\calH(G)$, let
\[
v(H):=\#\{\text{isolated vertices of }H\},\qquad
e(H):=\#\{\text{edge-components of }H\},
\]
and for each $k\ge 3$ let $c_k(H)$ denote the number of $k$-cycle components of $H$.
Then, if $G$ has $n$ vertices,
\begin{equation}\label{eq:sesq-balance}
v(H)+2e(H)+\sum_{k\ge 3}k\,c_k(H)=n.
\end{equation}
Let us represent the total number of cycles of a sesquivalent graph by $c(H)$, namely 
\[
c(H):=\sum_{k\ge 3}c_k(H).
\]

\begin{definition}[Sesquivalent polynomial]\label{def:phi}
We define the multivariate polynomial
\begin{equation}\label{eq:phi}
\Phi_G(x,y,z)\;:=\;\sum_{H\in\calH(G)} x^{v(H)}\,y^{e(H)}\,z^{c(H)}\qquad (x,y,z\in\CC).
\end{equation}
\end{definition}

The polynomial $\Phi_G(x,y,z)$ is a generalization of two well-studied graph polynomials, and shares several properties common to both, as we briefly comment below:

\paragraph{Characteristic polynomial.}
Given a finite simple graph $G=(V,E)$, the Harary--Sachs expansion expresses the coefficients of $\phi_G(\lambda)=\det(\lambda I-A_G)$ in terms of sesquivalent subgraphs. In particular, setting $y=-1$ and $z=-2$ gives $\Phi_G(\lambda,-1,-2)=\phi_G(\lambda)$ (see \cite{Harary62,Sachs66}). If $G$ has maximum degree $\Delta$, then $\Phi_G(\lambda,-1,-2)\neq 0$ whenever $|\lambda|>\Delta$.

\paragraph{Matching polynomial.}
Setting $z=0$ suppresses all cycle-components and reduces \eqref{eq:phi} to a monomer--dimer/matching-type partition function (see \cite{HLmonomerdimer}).
For the specialization $y=-1$ we obtain the matching polynomial
\[
\mu_G(x)=\sum_{k\ge 0}(-1)^k m_k(G)\,x^{n-2k},
\]
where $m_k(G)$ is the number of matchings of size $k$.
By the Heilmann--Lieb theorem \cite{HLmonomerdimer}, $\mu_G$ has only real zeros and all zeros lie in
$[-2\sqrt{\Delta-1},\,2\sqrt{\Delta-1}]$.

\paragraph{Recurrences and reconstruction.} The polynomial $\Phi_G(x,y,z)$ satisfies
\[
\frac{\partial}{\partial x} \Phi_G(x,y,z) = \sum_{a \in V(G)} \Phi_{G\setminus a}(x,y,z),
\]
and analogous expressions for the other variables. From the expression above, it follows that $\Phi_G(x,y,z)$ is reconstructible from the deck of vertex-deleted subgraphs of $G$ (the proof is a straightforward generalization of known cases, see \cite[Section 4.5]{godsil2017algebraic} for instance). It is also possible to derive recurrences for $\Phi_G(x,y,z)$ assuming $G$ is disconnected, or contains a cut-edge or a cut-vertex (along the lines of known expressions for $\phi_G$ and $\mu_G$, as in \cite[Sections 1.1 and 2.1]{godsil2017algebraic}).

\medskip

In this paper we show that, for any finite graph of bounded maximum degree, the sesquivalent polynomial can be interpreted (up to an explicit multiplicative factor) as the partition function of a hard-core polymer gas, in the standard abstract polymer-gas framework of statistical mechanics (see \cite{GruberKunz71, ScottSokal2005, FernandezProcacci07}). Partition functions of this type are central objects in both statistical physics and combinatorics, since their analytic properties, most notably the existence of explicit zero-free regions, encode absence of phase transitions and, in turn, enable powerful tools such as convergent cluster expansions (\cite{KoteckyPreiss1986, Dobrushin96, FernandezProcacci07})  and deterministic approximation schemes (\cite{Barv16, PatReg19, bencs2025barvinoksinterpolationmethodmeets}).

Our attention is on regime $z\neq 0$, since as mentioned above, for $z=0$ polynomial $\Phi_G(x,y,z)$ recovers the classical monomer--dimer/matching specialization, for which sharp results are known and  a general polymer-expansion bound is not expected to compete with the results on this axis. For $z\neq 0$  we develop a \emph{genuinely multivariate} analytic theory, by applying the Fern\'andez--Procacci criterion \cite{FernandezProcacci07}, the best convergence criterion for the abstract polymer gas, to derive an explicit zero-free region in \((x,y,z)\) for bounded-degree graphs, including a nontrivial regime where the cycle parameter \(z\) is not merely infinitesimal.

\begin{theorem}[Zero-free region for $\Phi_G$]\label{thm:main}
Let $G$ be a finite graph with maximum degree $\Delta\ge 2$.
Assume $x\in\CC$ satisfies
\begin{equation}\label{eq:x-cond}
|x|>(\Delta-1)e^{a}\qquad\text{for some }a>0.
\end{equation}
Then $\Phi_G(x,y,z)\neq 0$ whenever $(y,z)\in\CC^2$ satisfy
\begin{equation}\label{eq:main-ineq}
|y|+\frac{(\Delta-1)e^{a}}{|x|-(\Delta-1)e^{a}}\,|z|
\;\le\;
\alpha_\Delta\,(e^{a}-1),
\qquad
\alpha_\Delta:=\frac{(\Delta-1)^2}{\Delta}.
\end{equation}
\end{theorem}

As a computational consequence, working strictly inside this region yields a zero-free complex neighborhood along a canonical interpolation, and hence a deterministic Taylor/Barvinok-type approximation scheme for \(\Phi_G(x,y,z)\) (and for derivatives of \(\log \Phi_G\), i.e., densities of dimers and cycle-components) on graphs of maximum degree \(\Delta\).

This work is organized as follows: in Section \ref{sec:polymer} we relate $\Phi_G$ to an abstract polymer gas and prove Theorem \ref{thm:main}. In Section \ref{sec:barvinok} we provide a deterministic approximation for $\Phi$ à la Barvinok-Patel-Regts.

\section{Polymer-gas representation and proof of the zero-free region}\label{sec:polymer}

Models that combine isolated vertices (monomers) and copies of $K_2$ (dimers) with cycles have also been studied in statistical mechanics setting, e.g., the monomer--dimer--loop (MDL) model introduced by Li, Li, and Chen on the square lattice \cite{LiLiChen15MDL}. 
Our setting is different in two key respects: (i) we work on arbitrary finite graphs of bounded maximum degree (not a fixed lattice), and (ii) the additional objects are \emph{simple cycle-components} (vertex-disjoint cycles, as connected components of the chosen subgraph), weighted by a \emph{cycle parameter} \(z\) per cycle-component (rather than, e.g., an energy/fugacity per occupied loop bond as in lattice MDL models \cite{LiLiChen15MDL}).

Our strategy will be to relate this “cycle-component” weighting to the abstract polymer-gas, a discrete model proposed by Kotecky-Preiss \cite{KoteckyPreiss1986}. This model is defined by a triple
\((\calP, \omega, W)\), where $\calP$ is a countable set whose elements are called polymers, \(\omega : \calP \to  \CC \) is
a function that assigns to each polymer \(\gamma \in \calP\) a complex number \(\omega(\gamma)\), called the activity of $\gamma$ and \(W : \calP \times \calP \to \{0, 1\}\) is a function, called the Boltzmann factor, satisfying
$W(\gamma, \gamma) = 0$ for all $\gamma \in \calP$ and $W(\gamma, \gamma') = W(\gamma', \gamma)$ for all \(\{\gamma, \gamma'\} \in \calP\). Usually the pair
$\{\gamma, \gamma'\}$ is called \textit{incompatible} when $W(\gamma, \gamma') = 0$ and \textit{compatible} if $W(\gamma, \gamma') = 1$.
The partition function of this model can be written as 

\begin{equation}\label{eq:partfunction}
\Xi_{\calP}(\boldsymbol{\omega}) := \sum_{X \subseteq \mathcal{P}}  \prod_{\{\gamma, \gamma'\} \subseteq X} W(\gamma,\gamma')\prod_{\gamma \in X} \omega(\gamma).
\end{equation}
The best criterion for zero-free region of $\Xi_{\calP}(\boldsymbol{\omega})$ is known as  Fern\'andez–Procacci convergence criterion \cite{FernandezProcacci07}, as we will see below.


\subsection{From sesquivalent subgraphs to a polymer partition function}
Consider a graph $G=(V,E)$ and recall the sesquivalent polynomial defined in \ref{eq:phi}
\[
   \Phi_G(x,y,z) = \sum_{H \in \mathcal{H}(G)} y^{e(H)} z^{c(H)} x^{v(H)}.
\]
Fix $x\neq 0$. By Equation~\eqref{eq:sesq-balance}, each term in \eqref{eq:phi} can be rewritten as
\begin{align}\label{eq:rewrite}
x^{v(H)}y^{e(H)}z^{c(H)}
&=x^{\,n-2e(H)-\sum_{k\ge 3}k\,c_k(H)}\,y^{e(H)}\,z^{\sum_{k\ge 3}c_k(H)}\nonumber\\
&=x^n\Bigl(\frac{y}{x^2}\Bigr)^{e(H)}\prod_{k\ge 3}\Bigl(\frac{z}{x^k}\Bigr)^{c_k(H)}.
\end{align}
Hence, we can rewrite 

\begin{equation}\label{eq: xyz2}
\begin{aligned}
\Phi_G(x,y,z)
& = x^n \sum_{H\in\mathcal{H}(G)}\left[\left(\frac{y}{x^2}\right)^{e(H)} \prod_{k \ge 3} \left(\frac{z}{x^k 
 }\right)^{c_k(H)} \right]\\
 &=  x^n \sum_{H \in \calH(G)} \prod_{\substack{e \in H \\ e \, \text{is copy of}\, K_2 }} \omega(e) \prod_{\substack{C_k \in H \\ C_k \, \text{is $k$-cycle}}} \omega(C_k),
 \end{aligned}
\end{equation}
where, for any edge $e \in E$ and any $k$-cycle $C_k$ subgraph of $G$, we define its weight, respectively, as 
\begin{equation}\label{eq:actsesq}
\omega(e)= \frac{y}{x^2} \qquad \text{and} \qquad \omega(C_k)= \frac{z}{x^k 
 }.
\end{equation}
Note that a subgraph $H \in {\calH} (G)$ can be seen as a disjoint collection \(\{\gamma_1, \ldots, \gamma_m\}\), with $m \geq 0$, where each $\gamma_i$, for $i=1, \cdots, m$, represents a  connected component of $H$ (i.e., each $\gamma_i$ is either an edge or a cycle), then 
\begin{equation}\label{eq:xyzpartition}
\begin{aligned}
\Phi_G(x,y,z)= x^n \Xi_{\calH(G)}(\omega),
\end{aligned}
\end{equation}
where
\begin{equation}\label{eq: xyzpolymer}
\Xi_{\calH(G)}(\omega)= \sum_{m \geq 0} \sum_{\substack{H \in \calH(G)\\ H=\{\gamma_1, \ldots, \gamma_m\} \\ \gamma_i \cap \gamma_j = \emptyset } }\prod_{i=1}^{m} \omega(\gamma_i).
\end{equation}
Therefore, for $x\neq0$, a zero-free region for \(\Xi_{\calH(G)}(\omega)\) results in a region free of zeros for \(\Phi_G(x,y,z)\). 

Observe then that $\Xi_{\calH(G)}(\omega)$ can be seen as a hard-core polymer gas, with

\paragraph{Polymers.} The set of polymers $\calP $
is composed  of all edges $\{u,v\}\in E$ (with support $V(\gamma)=\{u,v\}$) and all \emph{simple cycles} $C$ in $G$ (with support $V(\gamma)=V(C)$).

\paragraph{Compatibility.}
 $W(\gamma, \gamma') = 0$ if $V(\gamma)\cap V(\gamma')\neq \varnothing$ and  $W(\gamma, \gamma') = 1$ otherwise. I.e., two polymers are incompatible if they share a vertex.

\paragraph{Activities.}
 the activities $\omega(\gamma)$ are given by~\ref{eq:actsesq}.
\medskip

We use the following standard (uniform) version of the Fern\'andez--Procacci criterion for abstract polymer gases; in the hard-core case (incompatibility by overlap of supports) it is often presented in a vertex-anchored form.

\begin{theorem}[Fern\'andez--Procacci criterion, hard-core case \cite{FernandezProcacci07}]\label{thm:FP}
Consider a hard-core polymer gas with polymer set $\calP$, incompatibility given by overlap of supports, and complex activities $\omega(\gamma)$.
If there exists $a>0$ such that
\begin{equation}\label{eq:FP}
\sup_{v\in V}\ \sum_{\substack{\gamma\in\calP\\ v\in V(\gamma)}} |\omega(\gamma)|\,e^{a|V(\gamma)|}
\ \le\ e^{a}-1,
\end{equation}
then the cluster expansion for $\log\Xi(\omega)$ converges absolutely and $\Xi(\omega)\neq 0$ at these activities.
\end{theorem}

\subsection{Proof of the main theorem}

\begin{proof}[Proof of Theorem \ref{thm:main}]
Fix a vertex $v\in V$. We bound the left-hand side of \eqref{eq:FP} by separating edge-polymers and cycle-polymers. Observe that we have at most $\Delta$ edges incident to $v$.
On the other hand, a convenient (but not sharp) bound for
the number of $k$-cycles containing $v$ is 
$ \Delta\,(\Delta-1)^{\,k-2}\,,\, k\ge 3\,$. Therefore, we can bound

\begin{equation}\label{eq:FP-sum} 
\begin{aligned}
\sum_{\substack{\gamma \in \calP \\ v \in  V(\gamma) }} |\omega(\gamma)|e^{a|V(\gamma)|} 
& \le \sum_{\substack{\gamma \in E\\ v \in  V(\gamma)}} \frac{|y|}{|x|^2}e^{2a} + \sum_{k \ge 3}\sum_{\substack{\gamma \text{is a $k$-cycle} \\ v \in  V(\gamma)}} \frac{|z|}{|x|^k}e^{ak}\\
& \le \frac{|y|}{|x|^2}e^{2a} \Delta + \sum_{k \ge 3}\Delta\,(\Delta-1)^{\,k-2}\frac{|z|}{|x|^k}e^{ak}\\
& = \frac{|y|}{|x|^2}e^{2a} \Delta +\frac{ |z| \Delta}{(\Delta-1)^{2}}\sum_{k \ge 3}{\left[\frac{(\Delta-1)e^{a}}{|x|}\right]}^k\\
& = \frac{|y|}{|x|^2}e^{2a} \Delta +\frac{ |z| \Delta}{(\Delta-1)^{2}} \left[\frac{(\Delta-1)e^{a}}{|x|}\right]^3 \frac{|x|}{|x|-{(\Delta-1)e^{a}}}\\
& = \frac{|y|}{|x|^2}e^{2a} \Delta +{ |z| \Delta} \left[\frac{(\Delta-1) e^{3a}}{|x|^2}\right] \frac{1}{|x|-{(\Delta-1)e^{a}}}\\
& = \frac{\Delta e^{2a}}{|x|^{2}}\left[ |y|  +  \frac{|z|(\Delta-1) e^{a}}{|x|-{(\Delta-1)e^{a}}}\right]\\
& \leq \frac{\Delta }{(\Delta -1)^{2}} \left[ |y|  +  \frac{|z|(\Delta-1) e^{a}}{|x|-{(\Delta-1)e^{a}}}\right],
\end{aligned}
\end{equation} 
assuming that $|x| > (\Delta-1)e^a$.

Therefore, defining \(\alpha_\Delta=\frac{(\Delta-1)^2}{\Delta}\), by Theorem \ref{thm:FP}, $\Xi(\omega)\neq 0$ whenever 
\[
 \frac{\Delta }{(\Delta -1)^{2}} \left[ |y|  +  \frac{|z|(\Delta-1) e^{a}}{|x|-{(\Delta-1)e^{a}}}\right]\ \le\ e^{a}-1,
 \]
and since $\Phi_G(x,y,z)=x^n\Xi(\omega)$ with $x\neq 0$, we conclude $\Phi_G(x,y,z)\neq 0$ as claimed.
\end{proof}
\medskip

\begin{remark}[Girth refinement]\label{rem:girth}
If $G$ has girth at least $\mathfrak g\ge 3$, graph $G$ possesses cycles $C_k$ only for  $ k \ge \mathfrak g$.
The cycle contribution in~\eqref{eq:FP-sum} can be replaced by
\[
\sum_{k \ge \mathfrak g}\sum_{\substack{\gamma~\text{is a $k$-cycle} \\ v \in  V(\gamma)}} \frac{|z|}{|x|^k}e^{ak}\\
\le
\frac{\Delta|z|}{(\Delta-1)^2}\cdot\frac{r^{\mathfrak g}}{1-r},
\]
instead of $\frac{\Delta|z|}{(\Delta-1)^2}\cdot\frac{r^3}{1-r}$, improving the cycle bound by a factor $r^{\mathfrak g-3}$, where $r=\frac{(\Delta-1)e^{a}}{|x|}<1$.
\end{remark}

\paragraph{Ex.} 
Specializing $\Phi_G(x,y,z)$ to $y=-1$ gives the explicit ``small cycle activity'' radius
\[
|z|\ \le\ z_{\max}(x;a)
:= \frac{|x|-(\Delta-1)e^{a}}{(\Delta-1)e^{a}}
\Bigl(\alpha_\Delta\,(e^{a}-1)-1\Bigr),
\]
which is positive provided
\[
\alpha_\Delta\,(e^{a}-1)>1
\quad\Longleftrightarrow\quad
a>\log\!\Bigl(1+\frac{1}{\alpha_\Delta}\Bigr)
=\log\!\Bigl(1+\frac{\Delta}{(\Delta-1)^2}\Bigr).
\]
Thus, for any $a$ above this threshold and any $x$ with $|x|>(\Delta-1)e^{a}$, turning on a cycle parameter $z$ with $|z|\le z_{\max}(x;a)$ preserves nonvanishing of $\Phi_G(x,-1,z)$.

Equivalently,  in the $x$-plane:
\[
|x|\ \ge\ (\Delta-1)e^{a}\Bigl(1+\frac{|z|}{\alpha_\Delta(e^{a}-1)-1}\Bigr)
\quad\Longrightarrow\quad
\Phi_G(x,-1,z)\neq 0.
\]
In particular, for each fixed $z$ in the above regime, all zeros of the univariate polynomial
$x\mapsto \Phi_G(x,-1,z)$ are confined to the disk
\[
|x|<(\Delta-1)e^{a}\Bigl(1+\frac{|z|}{\alpha_\Delta(e^{a}-1)-1}\Bigr),
\]

The bounds above depend on $y$ only through $|y|$. In particular, for both $y=1$ (positive-weight loop--dimer gas) and $y=-1$ (matching-like cancellations) we have $|y|=1$ and the above example shows that the zeros cannot drift arbitrarily far when one adds a sufficiently small activity of vertex-disjoint cycles.
{(We do not claim real-rootedness persists for $z\neq 0$; rather, we obtain a robust complex zero-free region and holomorphy of $\log \Phi_G$.)}

\medskip
The auxiliary parameter $a$ can be optimized. Indeed, rearranging \eqref{eq:main-ineq} for $|z|$ gives
\begin{equation}\label{eq:zbound-a}
|z|\ \le\ \frac{|x|-(\Delta-1)e^{a}}{(\Delta-1)e^{a}}\,
\Bigl(\alpha_\Delta(e^{a}-1)-|y|\Bigr),
\end{equation}
valid when $|x|>(\Delta-1)e^{a}$ and $\alpha_\Delta(e^{a}-1)>|y|$. Set
\[
c:=\frac{|x|}{\Delta-1},\qquad t:=e^{a},
\]
so $1<t<c$ and $t>1+\frac{|y|}{\alpha_\Delta}$.
Then \eqref{eq:zbound-a} becomes
\begin{equation}\label{eq:g-def}
|z|\ \le\ g(t)
:=\Bigl(\frac{c}{t}-1\Bigr)\Bigl(\alpha_\Delta(t-1)-|y|\Bigr),
\qquad 1+\frac{|y|}{\alpha_\Delta}<t<c.
\end{equation}
Note that pushing $a$ close to the boundary value $\log c$ (i.e.\ $t\approx c$) makes the prefactor $\frac{c}{t}-1$ very small, hence is not optimal when the goal is to allow a nontrivial $|z|$. This reflects the trade-off between enlarging $e^{a}-1$ and keeping the multiplicative factor $\frac{|x|-(\Delta-1)e^{a}}{(\Delta-1)e^{a}}$ away from~$0$.

\begin{lemma}\label{lem:optimal-a}
Fix $\Delta$ and $|y|$, and assume the interval in \eqref{eq:g-def} is nonempty.
Then $g(t)$ is maximized at
\begin{equation}\label{eq:tstar}
t_\ast=\sqrt{c\Bigl(1+\frac{|y|}{\alpha_\Delta}\Bigr)},
\qquad\text{equivalently}\qquad
a_\ast=\log t_\ast
=\frac12\log\!\Bigl(\frac{|x|}{\Delta-1}\Bigr)
+\frac12\log\!\Bigl(1+\frac{|y|}{\alpha_\Delta}\Bigr).
\end{equation}
\end{lemma}

\begin{proof}
Expanding \eqref{eq:g-def} gives
\[
g(t)= -\alpha_\Delta t - \frac{c(\alpha_\Delta+|y|)}{t} + c\alpha_\Delta + (\alpha_\Delta+|y|).
\]
Thus
\[
g'(t)=-\alpha_\Delta + \frac{c(\alpha_\Delta+|y|)}{t^2},
\qquad
g''(t)=-\frac{2c(\alpha_\Delta+|y|)}{t^3}<0,
\]
so the unique critical point $t_\ast^2=\frac{c(\alpha_\Delta+|y|)}{\alpha_\Delta}=c(1+\frac{|y|}{\alpha_\Delta})$ is the unique maximizer.
\end{proof}

\begin{proposition}[Linear admissible cycle parameter]\label{prop:linear-z}
Fix $\Delta$ and $|y|$.
Assume $c=\frac{|x|}{\Delta-1}$ is large enough so that $t_\ast$ from \eqref{eq:tstar} lies in the admissible interval of \eqref{eq:g-def}
(e.g.\ it suffices that $c>\bigl(1+\frac{|y|}{\alpha_\Delta}\bigr)^2$).
Then \eqref{eq:g-def} holds with $a=a_\ast$ and
\begin{equation}\label{eq:zmax-exact}
|z|\ \le\ g(t_\ast)
=\alpha_\Delta c+(\alpha_\Delta+|y|)-2\alpha_\Delta\,t_\ast.
\end{equation}
In particular, as $|x|\to\infty$ with $\Delta$ and $|y|$ fixed,
\begin{equation}\label{eq:zmax-asymp}
|z|\ \le\ \frac{\Delta-1}{\Delta}\,|x|\ -\ O(\sqrt{|x|}).
\end{equation}
\end{proposition}

\begin{proof}
Using $t_\ast^2=\frac{c(\alpha_\Delta+|y|)}{\alpha_\Delta}$, we have $\frac{c(\alpha_\Delta+|y|)}{t_\ast}=\alpha_\Delta t_\ast$.
Substituting into the expanded form of $g(t)$ yields \eqref{eq:zmax-exact}.
Finally, $\alpha_\Delta c=\frac{(\Delta-1)^2}{\Delta}\cdot\frac{|x|}{\Delta-1}=\frac{\Delta-1}{\Delta}|x|$, while
\[
t_\ast=\sqrt{c\Bigl(1+\frac{|y|}{\alpha_\Delta}\Bigr)}
=\sqrt{\frac{|x|}{\Delta-1}}\;\sqrt{1+\frac{|y|}{\alpha_\Delta}},
\]
 giving \eqref{eq:zmax-asymp}.
\end{proof}

\begin{remark}[Interpretation: cycles are small relative to their size]\label{rem:interpret}
Even if $|z|\le M|x|$ for some constant $M>0$, the activity of a single $k$-cycle polymer satisfies
\[
\bigl|\omega(C_k)\bigr|=\frac{|z|}{|x|^{k}}
\le \frac{M|x|}{|x|^k}
=\frac{M}{|x|^{k-1}}.
\]
In particular, for every fixed $k\ge 3$ this tends to $0$ as $|x|\to\infty$.
Thus the regime in Proposition~\ref{prop:linear-z} is still a low-density regime at the polymer level: larger polymers are strongly suppressed by their size, and the cluster expansion converges despite a cycle parameter $z$ of linear order in $|x|$.
\end{remark}


\section{Barvinok interpolation and deterministic approximation}\label{sec:barvinok}

In this section we use the zero-free region from Theorem~\ref{thm:main} to develop an explicit deterministic approximation procedure for $\Phi_G(x,y,z)$, when $G$ is a graph with maximum degree at most $\Delta$. Our strategy is to combine Barvinok’s method \cite{Barv16} with the bounded-degree interpolation framework of Patel--Regts \cite{PatReg19}.

\subsection{Analytic reduction: zero-freeness and Taylor truncation}\label{subsec:analytic}

We start by reducing the evaluation of $\Phi_G$ to a univariate interpolation $F(t)$ as follows.
Fix $(x,y,z)\in\CC^3$ with $x\neq 0$, and define
\begin{equation}\label{eq:def-F2}
F(t)\; :=\;t^{n}\,\Phi_G\!\left(\frac{x}{t},\,y,\,z\right),
\end{equation}
with $t\in\CC$. Note that $F(t)$ is a polynomial of degree at most $n$, with $F(0)=x^{n}\neq 0$ and $F(1)=\Phi_G(x,y,z)$.

If $F$ is zero-free on $|t|\le\rho$ with $\rho>1$, then there exists an analytic branch of the logarithm on this disk.
In particular, there is an analytic function $\log F(t)$ on $|t|\le\rho$ such that
\[
\exp(\log F(t))=F(t)\quad\text{for all }|t|\le\rho\]
and we fix the branch by requiring that $(\log F)(0)$ equals $\log(F(0))$.
Moreover, $\log F(t)$  admits a convergent Taylor expansion
\[
\log F(t)=\sum_{k\ge0} b_k t^k \quad \text{for}\, \quad |t|<\rho.
\]

By Barvinok’s method \cite{Barv16}, if the univariate interpolation $F(t)$ is zero-free on a disk $|t|\le \rho$ with $\rho>1$, then $\log F$ is analytic on this disk and $\log F(1)$ can be approximated (with an explicit, controllable error) by truncating the Taylor series of $\log F(t)$ at $t=0$.

Write the constants from Theorem~\ref{thm:main} as
\[
c:=(\Delta-1)e^{a},
\qquad
B:=\alpha_\Delta (e^{a}-1),
\qquad
\alpha_\Delta=\frac{(\Delta-1)^2}{\Delta}.
\]
Assume that $(x,y,z)$ is strictly inside the region of Theorem~\ref{thm:main}, i.e.
\begin{equation}\label{eq:strict-slack2}
|x|>c
\quad\text{and}\quad
|y|+\frac{c}{|x|-c}\,|z| \;<\; B.
\end{equation}
Define the slack
\begin{equation}\label{eq:def-slack2}
\delta \;:=\; B-|y| \;>\;0.
\end{equation}

\begin{lemma}[Explicit disk]\label{lem:disk2}
Under \eqref{eq:strict-slack2}, set
\begin{equation}\label{eq:def-rho2}
\rho \;:=\; \frac{|x|}{c\left(1+\frac{|z|}{\delta}\right)}.
\end{equation}
Then $\rho>1$ and $F(t)\neq 0$ for all $t\in\CC$ with $|t|\le \rho$.
\end{lemma}

\begin{proof}
If $t=0$, then $F(0)=x^n\neq 0$. Fix $t$ with $0<|t|\le \rho$.
By definition of $\rho$,
\[
\left|\frac{x}{t}\right|\ge \frac{|x|}{\rho} = c\left(1+\frac{|z|}{\delta}\right) > c,
\]
so the $x$-condition in Theorem~\ref{thm:main} holds for $x/t$. Moreover,
\[
\left|\frac{x}{t}\right|-c \;\ge\; \frac{|x|}{\rho}-c
= \frac{c|z|}{\delta}.
\]
If $z=0$, then $\frac{c}{|x/t|-c}|z|=0$. If $z\neq 0$, then
\[
\frac{c}{|x/t|-c}\,|z|
\;\le\;
\frac{c}{c|z|/\delta}\,|z|
\;=\;\delta.
\]
Using \eqref{eq:def-slack2}, we obtain
\[
|y|+\frac{c}{|x/t|-c}\,|z|
\;\le\;
|y|+\delta
= B,
\]
and Theorem~\ref{thm:main} yields $\Phi_G(x/t,y,z)\neq 0$. Since $F(t)=t^n\Phi_G(x/t,y,z)$ and $t\neq 0$,
we conclude that $F(t)\neq 0$ for all $0<|t|\le\rho$. Finally, $\rho>1$ follows from strictness in \eqref{eq:strict-slack2}, which is equivalent to $\frac{c}{|x|-c}|z|<\delta$,
i.e.\ $|x|>c\left(1+\frac{|z|}{\delta}\right)$, hence $\rho>1$.
\end{proof}

\begin{lemma}[Tail bound]\label{lem:tail2}
Suppose that $F(t)$ defined in~\eqref{eq:def-F2} is zero-free on $|t|\le\rho$ with $\rho>1$. Let $\log F$ denote the analytic logarithm on $|t| \leq \rho$ fixed as above, write  
\begin{equation}
  \log F(t)=\sum_{k\ge 0} b_k t^k \quad \text{and define \quad} T_m(t):=\sum_{k=0}^{m}b_kt^k.
\end{equation}
Given $\varepsilon$, we have that  $|\log F(1)-T_m(1)|\le \varepsilon$ as soon as
\begin{equation}\label{eq:m}
m \;:=\; \left\lceil \frac{\log\!\left(\frac{n}{(\rho-1)\varepsilon}\right)}{\log\rho}\right\rceil .
\end{equation}
\end{lemma}

\begin{proof}
Let us rewrite $F(t)$ as 
\begin{equation}\label{eq:factor}
F(t)=F(0)\prod_{i=1}^{d}\left(1-\frac{t}{\zeta_i}\right),
\qquad d=\deg F,
\end{equation}
where $\zeta_1,\dots,\zeta_d$ are the zeros of $F$ (listed with multiplicity) and then
\[
\log F(t)=\log F(0)+\sum_{i=1}^{d}\log\!\left(1-\frac{t}{\zeta_i}\right).
\]
Zero-freeness on $|t|\le\rho$ implies $|\zeta_i|\ge \rho$ for all $i$. Therefore, for $|t|\le 1$ we have $\left|\frac{t}{\zeta_i}\right|\le \frac{1}{\rho}<1$, so we may use the absolutely convergent expansion
$\log(1-w)=-\sum_{k\ge 1} \frac{w^k}{k}$.
Thus
\[
\log F(t)=\log F(0)-\sum_{i=1}^{d}\sum_{k\geq1}\left(\frac{t^k}{k\zeta_i^k}\right)=\log F(0)-\sum_{k\geq1}\frac{t^k}{k}\left(\sum_{i=1}^{d}\frac{1}{\zeta_i^k}\right).
\]
By comparing coefficients we have $b_0=\log F(0)$ and, for $k\ge1$,
\[
b_k=-\frac{1}{k}\sum_{i=1}^{d}\zeta_i^{-k}.
\]
Hence 
\[
\left|\log F(1)-T_m(1)\right|
\le
\sum_{k>m}\frac{1} {k}\left(\sum_{i=1}^{d}\frac{1}{|\zeta_i|^k}\right)
\le
\sum_{k>m}\frac{d}{k\rho^k}.
\]
Then 
\[
\left|\log F(1)-T_m(1)\right|
\le
\sum_{k>m}\frac{d}{k\rho^k}
\le
\frac{d}{m+1}\sum_{k>m}\rho^{-k}
=
\frac{d}{(m+1)}\cdot\frac{\rho^{-m}}{\rho-1}
=
\frac{d}{(m+1)(\rho-1)\rho^{m}}.
\]
Since $\deg F=d\le n$, we further obtain
\[
\left|\log F(1)-T_m(1)\right|
\le
\frac{n}{(m+1)(\rho-1)\rho^{m}}.
\]
Taking \[
m \;=\; \left\lceil \frac{\log\!\left(\frac{n}{(\rho-1)\varepsilon}\right)}{\log\rho}\right\rceil,
\]
then  $\rho^m \ge \frac{n}{(\rho-1)\varepsilon}$, hence
\[
\left|\log F(1)-T_m(1)\right|\leq \frac{n}{(m+1)(\rho-1)\rho^{m}}
\le
\frac{n}{(m+1)(\rho-1)}\cdot\frac{(\rho-1)\varepsilon}{n}
=
\frac{\varepsilon}{(m+1)}
\le \varepsilon,
\]
which completes the proof.
\end{proof}

Combining Lemmas~\ref{lem:disk2} and \ref{lem:tail2}, we reduce the approximation of
$F(1)=\Phi_G(x,y,z)$ to computing the first $m$ coefficients $b_0,\dots,b_m$ of the Taylor expansion
$\log F(t)=\sum_{k\ge0} b_k t^k$ at $t=0$, where $m$ is given by~\eqref{eq:m}.
In the next step we show that, for bounded-degree graphs, these coefficients can be computed deterministically
within the Patel--Regts bounded-degree interpolation framework.

\subsection{Algorithmic step: computing the Taylor coefficients via Patel--Regts}\label{subsec:algorithmic}

It remains to compute the coefficients $b_0,\dots,b_m$ of the Taylor expansion
$\log F(t)=\sum_{k\ge0} b_k t^k$ at $t=0$, where $m$ is given by~\eqref{eq:m}.
We do this using the bounded-degree interpolation framework of Patel--Regts~\cite{PatReg19}.

For fixed $(x,y,z)$, we define 
\[
 \widetilde F_G(t):=\frac{F_G(t)}{x^n},
\]
then $\widetilde F_G(0)=1$ and on the zero-free disk of Lemma~\ref{lem:disk2} we fix the analytic branch of $\log F_G$
by requiring $\log F_G(0)=\log(x^n)$. In particular,
\[
\log F_G(t)=\log(x^n)+\log \widetilde F_G(t),
\]
where $\log \widetilde F_G$ is the analytic logarithm normalized by $\log \widetilde F_G(0)=0$. Then
 $b_0=\log(x^n)$ and for every $k\ge1$,
\[
b_k=[t^k]\log \widetilde F_G(t).
\]

From the definition of $F_G(t)$,
\begin{equation}\label{eq:def-F3-rev}
F_G(t)=\sum_{H\in\calH(G)} x^{v(H)}y^{e(H)}z^{c(H)}\, t^{\,n-v(H)}.
\end{equation}
Hence, writing $\widetilde F_G(t)=\sum_{k=0}^{n} a_k(G)\,t^k$, we have
\begin{equation}\label{eq:ak2-rev}
a_k(G)=[t^k]\widetilde F_G(t)
= x^{-k} \sum_{\substack{H\in\calH(G)\\ n-v(H)=k}} y^{e(H)}\,z^{c(H)}.
\end{equation}

Given a vertex set $U \subseteq V(G)$, let $G[U]$ denote the induced subgraph on $U$.
For a graph $J$ on $k$ vertices, let $\mathcal H(J)$ denote the set of sesquivalent spanning subgraphs of $J$.
Define
\[
\mathcal S(J):=\{S\in\mathcal H(J): v(S)=0\},
\qquad
\lambda(J):=x^{-k}\sum_{S\in\mathcal S(J)} y^{e(S)}z^{c(S)}.
\]

\begin{lemma}\label{lem:bigcp-form}
For every $k\in\{0,1,\dots,n\}$,
\begin{equation}\label{eq:ak3-rev}
a_k(G)=\sum_{\substack{U\subseteq V(G)\\ |U|=k}} \lambda(G[U]).
\end{equation}
\end{lemma}

\begin{proof}
Fix $k$. Recall from~\eqref{eq:ak2-rev} that
\[
a_k(G)=x^{-k}\sum_{\substack{H\in\mathcal H(G)\\ n-v(H)=k}} y^{e(H)}z^{c(H)}.
\]
Let $H\in\mathcal H(G)$ with $n-v(H)=k$ and let $U\subseteq V(G)$ be the set of non-isolated vertices of $H$, i.e., $|U|=k$. Consider the induced subgraph $G[U]$ and let $S:=H[U]$ (restriction of $H$ to $U$).
Since all isolated vertices of $H$ lie in $V(G)\setminus U$, the subgraph $S$ has no isolated vertices, i.e.\ $v(S)=0$.
Moreover, $S$ is sesquivalent on the vertex set $U$, hence $S\in\mathcal S(G[U])$.

Conversely, given $U\subseteq V(G)$ with $|U|=k$ and $S\in\mathcal S(G[U])$, there is a unique
$H\in\mathcal H(G)$ obtained by taking $H[U]=S$ and setting every vertex in $V(G)\setminus U$ as isolated vertices.
This establishes a bijection between sesquivalent subgraph $H$ with $n-v(H)=k$ and pairs $(U,S)$ with $|U|=k$ and $S\in\mathcal S(G[U])$.
Substituting into~\eqref{eq:ak2-rev} yields~\eqref{eq:ak3-rev}.
\end{proof}

\begin{remark}[Grouping by isomorphism]\label{rem:isomorphism}
Let $\mathcal G_k$ be a fixed set of representatives of isomorphism classes of graphs on $k$ vertices.
For $J\in\mathcal G_k$, let $\ind(J,G)$ be the number of vertex subsets $U\subseteq V(G)$ with $|U|=k$ such that $G[U]\cong J$.
Since $\lambda(G[U])$ depends only on the isomorphism class of $G[U]$, we can group the sum in~\eqref{eq:ak3-rev} and write
\begin{equation}\label{eq:BIGCP-expansion-rev}
a_k(G)=\sum_{J\in\mathcal G_k}\lambda_J\,\ind(J,G),
\qquad\text{where }\ \lambda_J:=\lambda(J)=x^{-k}\sum_{S\in\mathcal S(J)} y^{e(S)}z^{c(S)}.
\end{equation}
\end{remark}

Patel--Regts~\cite{PatReg19} introduce \emph{bounded induced graph counting polynomials} (BIGCPs) as graph polynomials
whose coefficients admit an expansion of the form~\eqref{eq:BIGCP-expansion-rev}, with weights computable in time
$\beta^{|V(J)|}$ for some constant $\beta$, and which are multiplicative under disjoint unions. 

\begin{lemma}\label{lem:bigcp}
Fix $(x,y,z)$ and $\Delta$. For graphs $G$ of maximum degree at most $\Delta$, the assignment $G\mapsto \widetilde F_G(t)$
is a BIGCP in the sense of \cite[Def.~3.1]{PatReg19}.
Moreover, for each $k$-vertex graph $J$ with $\Delta(J)\le \Delta$, the weight $\lambda_J$ in~\eqref{eq:BIGCP-expansion-rev}
can be computed in time $(2^{\Delta/2})^{k}\cdot \mathrm{poly}(k)$ (hence the BIGCP computability requirement holds with constant $\beta=2^{\Delta/2}$).
\end{lemma}

\begin{proof}
As mentioned above, Lemma~\ref{lem:bigcp-form} provides the coefficient expansion demanded by the definition of BIGCP. Now let us check the multiplicativity of the mapping $G\mapsto \widetilde F_G(t)$. Suppose that graph $G$ can be decomposed in the disjoint union $G=G_1\sqcup G_2$, then a sesquivalent subgraph of $G$ can also be decomposed uniquely
into a sesquivalent subgraph of $G_1$ and one of $G_2$. Therefore,  we obtain $\widetilde F_{G}(t)=\widetilde F_{G_1}(t)\,\widetilde F_{G_2}(t)$.

Finally we check the weights. Fix a graph $J$ on $k$ vertices with maximum degree at most $\Delta$.
By enumerating all spanning subgraphs $S=(V(J),F)$, $F\subseteq E(J)$ (there are $2^{|E(J)|}$ choices), and for each $F$
testing $S\in\mathcal S(J)$ and computing $e(S),c(S)$ in $\mathrm{poly}(k)$ time. Since $|E(J)|\le \Delta k/2$, this gives
time $2^{\Delta k/2}\mathrm{poly}(k)$, hence the weights satisfy the BIGCP computability requirement
\cite[Def.~3.1]{PatReg19} with $\beta=2^{\Delta/2}$.
\end{proof}

We now state explicitly the deterministic approximation algorithm obtained from the zero-free disk of
Lemma~\ref{lem:disk2} and the Taylor truncation bound of Lemma~\ref{lem:tail2}.

\begin{theorem}[Deterministic approximation for $\Phi_G(x,y,z)$]\label{thm:det-approx2}
Let $G$ be a graph on $n$ vertices with maximum degree at most $\Delta$.
Assume $(x,y,z)\in\CC^3$ satisfies the strict slack condition \eqref{eq:strict-slack2}, and let $\rho>1$ be the radius
defined in \eqref{eq:def-rho2}. Fix $\varepsilon\in(0,1)$ and set
\[
m \;:=\; \left\lceil \frac{\log\!\left(\frac{n}{(\rho-1)\varepsilon}\right)}{\log\rho}\right\rceil .
\]
There is a deterministic algorithm that, given $(G,x,y,z,\varepsilon)$, outputs a complex number $\widehat{\Phi}$
such that
\[
\bigl|\log \Phi_G(x,y,z)-\log \widehat{\Phi}\bigr|\le \varepsilon,
\qquad\text{equivalently}\qquad
\widehat{\Phi}=\Phi_G(x,y,z)\,e^{\eta}\ \text{with}\ |\eta|\le \varepsilon,
\]
where $\log$ denotes the analytic branch induced by the zero-freeness of $F_G$ on $|t|\le\rho$ with $\log F_G(0)=\log(x^n)$.
The running time is
\[
n\cdot (e\Delta)^{O(m)}\cdot \mathrm{poly}(m).
\]
\end{theorem}

\paragraph{Algorithm (Barvinok--Taylor on the canonical interpolation).}
\begin{enumerate}[label=\textup{(\arabic*)}, leftmargin=2.2em]
\item Form $F_G(t)=t^n\Phi_G(x/t,y,z)$ and compute $\rho>1$ as in Lemma~\ref{lem:disk2}.
\item Set $m=\left\lceil \frac{\log(\frac{n}{(\rho-1)\varepsilon})}{\log\rho}\right\rceil$.
\item Compute the inverse power sums $p_j=\sum_i \zeta_i^{-j}$ of the zeros $(\zeta_i)$ of $\widetilde F_G$ for $1\le j\le m$
      using \cite[Thm.~3.1]{PatReg19} (applicable by Lemma~\ref{lem:bigcp}).
\item Set $b_0=\log(x^n)$ and $b_j=-p_j/j$ for $1\le j\le m$, and output $\widehat{\Phi}:=\exp\!\bigl(\sum_{j=0}^{m} b_j\bigr)$.
\end{enumerate}

\begin{proof}
By Lemma~\ref{lem:disk2}, $F_G(t)\neq 0$ for all $|t|\le \rho$, hence $\log F_G$ is holomorphic on $|t|\le\rho$
(with the branch specified by $\log F_G(0)=\log(x^n)$) and admits a convergent Taylor expansion at $0$.
By Lemma~\ref{lem:tail2} and the choice of $m$,
\[
\bigl|\log F_G(1)-T_m(1)\bigr|\le \varepsilon,
\qquad
T_m(t):=\sum_{k=0}^{m} b_k t^k.
\]
Since $F_G(1)=\Phi_G(x,y,z)$, the output $\widehat{\Phi}=\exp(T_m(1))$ satisfies
$|\log \Phi_G(x,y,z)-\log \widehat{\Phi}|\le \varepsilon$, which is equivalent to
$\widehat{\Phi}=\Phi_G(x,y,z)\,e^{\eta}$ with $|\eta|\le \varepsilon$.

It remains to justify Step~(3) and the running time.
By Lemma~\ref{lem:bigcp}, $\widetilde F_G$ is a BIGCP for bounded-degree graphs. Therefore,
\cite[Thm.~3.1]{PatReg19} computes the inverse power sums $p_1,\dots,p_m$ of the zeros of $\widetilde F_G$ in time
$n\cdot (e\Delta)^{O(m)}\cdot \mathrm{poly}(m)$.

Finally, since $\widetilde F_G(0)=1$ we have on the zero-free disk
\[
\log \widetilde F_G(t)=\sum_i \log\!\left(1-\frac{t}{\zeta_i}\right)
=-\sum_{j\ge 1}\frac{p_j}{j}\,t^j,
\]
hence $[t^j]\log \widetilde F_G(t)=-p_j/j$ for $1\le j\le m$.
Using $\log F_G(t)=n\log x+\log \widetilde F_G(t)$ yields $b_0=\log(x^n)$ and $b_j=-p_j/j$ for $1\le j\le m$,
which is exactly Step~(4).
\end{proof}

\bibliographystyle{plainurl}
\bibliography{references}

\end{document}